\documentclass{elsarticle}

\usepackage{lineno,hyperref}
\modulolinenumbers[5]

\journal{Nonlinear Analysis, Theory, Methods \& Applications}

\usepackage[utf8]{inputenc}
\usepackage{amsmath}
\usepackage{amssymb}
\usepackage{mathrsfs}
\usepackage{graphics}
\usepackage{color}
\usepackage{amsthm}
\usepackage{paralist}
\usepackage{dsfont}
\usepackage{verbatim}
\usepackage{bm}
\allowdisplaybreaks

\renewcommand{\tilde}{\widetilde}
\renewcommand{\hat}{\widehat}
\renewcommand{\bar}{\overline}

\theoremstyle{definition}
\newtheorem{definition}{Definition}[section]
\newtheorem{remark}[definition]{Remark}
\newtheorem{assumption}[definition]{Assumption}
\theoremstyle{plain}
\newtheorem{lemma}[definition]{Lemma}
\newtheorem{thm}[definition]{Theorem}
\newtheorem{prop}[definition]{Proposition}

\numberwithin{equation}{section}

\newcommand{\dd}{\,\mathrm{d}}

\newcommand{\dff}{\mathrm{D}}

\newcommand{\auxil}{\mathcal{A}}
\newcommand{\metr}{\mathbf{d}}

\newcommand{\ent}{\mathcal{E}}

\newcommand{\M}{\mathbf{M}}
\newcommand{\mob}{\mathbf{m}}

\newcommand{\X}{\mathbf{X}}

\newcommand{\N}{{\mathbb{N}}}

\newcommand{\R}{{\mathbb{R}}}
\newcommand{\Rp}{{\mathbb{R}_{>0}}}

\newcommand{\flow}{\mathsf{S}}

\newcommand{\Matn}{\R^{n\times n}}
\newcommand{\tT }{\mathrm{T}}

\newcommand{\measm}{\mathscr{M}(\R;S)}

\newcommand{\inn}[1]{\mathrm{int}(#1)}

\newcommand{\hess}{\nabla^2_z}

\newcommand{\mom}[1]{\m_2(#1)}
\newcommand{\mss}[1]{\m_0(#1)}
\newcommand{\pot}{\mathcal{V}}
\newcommand{\zref}{{\bar{z}}}
\newcommand{\Xaux}{\mathbf{X}_{\zref}}

\newcommand{\gau}[1]{\left\lfloor#1\right\rfloor}
\newcommand{\ban}{\mathbf{Y}}
\newcommand{\dom}{\mathrm{Dom}}

\newcommand{\calH}{{\mathcal{H}}}

\newcommand{\W}{\mathbf{W}}
\newcommand{\m}{\bm{\mathfrak{m}}}

\newcommand{\lebm}[1]{\mathcal{L}\left(#1\right)}
\newcommand{\nonl}{\mathfrak{N}}
\newcommand{\loc}{\text{loc}}

\newcommand{\argmin}{\operatornamewithlimits{argmin}}

\begin{document}

\begin{frontmatter}
\author[tu]{Daniel Matthes\corref{corauth}}
\ead{matthes@ma.tum.de}
\author[tu]{Jonathan Zinsl}
\ead{zinsl@ma.tum.de}
\address[tu]{Zentrum f\"ur Mathematik \\ Technische Universit\"at M\"unchen \\ 85747 Garching, Germany}
\title{Existence of solutions for a class of fourth order cross-diffusion systems of gradient flow type\tnoteref{t1}}
\tnotetext[t1]{This research was supported by the DFG Collaborative Research Center TRR 109, ``Discretization in Geometry and Dynamics''.}
\cortext[corauth]{Corresponding author}

\begin{abstract}
  This article is concerned with the existence and the long time behavior of weak solutions 
  to certain coupled systems of fourth-order degenerate parabolic equations of gradient flow type.
  The underlying metric is a Wasserstein-like transportation distance for vector-valued functions,
  with nonlinear mobilities in each component.
  Under the hypothesis of (flat) convexity of the driving free energy functional,
  weak solutions are constructed by means of the variational minimizing movement scheme for metric gradient flows. 
  The essential regularity estimates are derived by variational methods.
\end{abstract}


\end{frontmatter}

%%%%%%%%%%%%%%%%%%%%%%%

\section{Introduction}\label{ch:dsintro}
In this article, we study the existence and the large-time behavior of vector valued solutions $u:\R_{\ge 0}\times\R\to S\subset\R^n$
to the following coupled system of nonlinear fourth-order equations in one spatial dimension:
\begin{align}\label{eq:equation4}
  \partial_t u + \partial_x\bigg(\M(u) \bigg[\partial^2_x\nabla_pf(\partial_xu,u)-\partial_x\nabla_zf(\partial_xu,u)\bigg]\bigg) = 0,
  % \begin{split}
  %   \partial_t u=\partial_x\bigg(\bigg.\M( u)\bigg[\bigg.&\nabla^2_{z}f( \partial_x u, u) \partial_x u-\partial_x(\nabla^2_{p}f(\partial_x u, u) \partial_{x}^2 u)\\&\quad-\partial_x(\nabla_{pz}f( \partial_x u, u)) \partial_x u\bigg.\bigg]\bigg.\bigg),
  % \end{split}
\end{align}
for $t>0$ and $x\in\R$,
subject to the initial condition $u(0,\cdot)=u^0$ for a given function $u^0:\R\to S$.
We refer to the cuboid 
\begin{align}
  \label{eq:cuboid}
  S=[S^\ell_1,S^r_1]\times\cdots\times[S^\ell_n,S^r_n]  \subset\R^n
\end{align}
with given lower-left and upper-right corners $S^\ell,S^r\in\R^n$, respectively, as \emph{value space}.
The \emph{mobility matrix} $\M(\cdot):S\to\R^{n\times n}$ is assumed to be \emph{fully decoupled} in the sense of \cite{zm2014},
that is:
\begin{align}
  \label{eq:fullydecoupled}
  \M(z) =
  \begin{pmatrix}
    \mob_1(z_1) & & & \\ & \mob_2(z_2) & & \\ & & \ddots & \\ & & & \mob_n(z_n)
  \end{pmatrix},
\end{align}
with $n$ nonnegative and concave (scalar) mobility functions $\mob_k:[S^\ell_k,S^r_k]\to\R$. 
For the precise assumptions on $\M$, see Definition \ref{def:H} below.
Finally, $f:\R^n\times S\to\R$ is a prescribed smooth \emph{free energy density} $f=f(p,z)$,
with properties that are specified in Assumption \ref{def:E} below.
Notice that the gradients $\nabla_z$ and $\nabla_p$ in \eqref{eq:equation4} 
act with respect to the $n$ components of $u$ and of $\partial_xu$, respectively.

Introducing the gradient-dependent free energy functional
\begin{align}\label{eq:energy}
  \ent(u)=\int_\R f(\partial_x u,u)\dd x,
\end{align}
it becomes apparent that \eqref{eq:equation4} is --- formally --- a gradient flow in the potential landscape of $\ent$,
\begin{align}
  \label{eq:pdesystem}
  \partial_t u=\partial_x\left[\M(u)\,\partial_x\frac{\delta\ent}{\delta u}(u)\right].
\end{align}
The main purpose of the paper at hand is to carry out a new approach 
--- alternative to the already existing ones --- 
to prove the existence of weak solutions to equations of type \eqref{eq:equation4}
% and to study their long time behavior
on the basis of the indicated gradient flow structure,
using variational methods.
The key to fit \eqref{eq:pdesystem} % with \eqref{eq:fullydecoupled}
into the theory of gradient flows in metric spaces 
are the recent results \cite{zm2014} concerning 
the induced metrics $\W_\M$ on the space $\measm$ of $S$-valued measurable functions.
A brief review of some essential properties of $\W_\M$ is provided in Section \ref{sct:WM}.
Gradient flows in $\W_\M$ for simpler functionals than \eqref{eq:energy} have been studied in \cite{zm2014},
following up on the results \cite{dns2009,lisini2010,carrillo2010} for scalar equations with nonlinear mobility functions.

The initial motivation for studying \eqref{eq:equation4} 
is its similarity to multi-component Cahn-Hilliard systems 
(see e.g. \cite{Gurtin} for a review on their derivation),
which are of the general form 
\begin{align}
  \label{eq:ch}
  \partial_tu-\partial_x\big(\widetilde{\M}(u)\partial_x\mu(u)\big) = 0, 
  \quad \mu(u)=-\Gamma\partial_x^2u+\nabla_z\Psi(u).
  % \partial_tu_j= -\partial_x\left[u_j(1-u_j)\,\partial_x\left(\sum_{\ell=1}^nQ_{j\ell}\,\partial_x^3u_\ell\right)\right] + \theta\,\partial_{xx}g_j(u_j).
  % \partial_t \begin{pmatrix}u_1\\u_2\\\vdots\\ u_n\end{pmatrix}&=-\partial_x\left[\begin{pmatrix}
  %     u_1(1-u_1) & & & \\ & u_2(1-u_2) & & \\ & & \ddots & \\ & & & u_n(1-u_n)
  %   \end{pmatrix}Q\partial_{x}^3\begin{pmatrix}u_1\\u_2\\\vdots\\ u_n\end{pmatrix}\right]\\%&\quad+\theta\partial_{x}^2\begin{pmatrix}u_1\\u_2\\\vdots\\ u_n\end{pmatrix},
\end{align}
Here $\Gamma\in\R^{n\times n}$ is a positive definite matrix, 
$\Psi:\tilde S\to\R$ is the homogeneous part of the free energy density,
and $\widetilde{\M}(\cdot):\tilde S\to\R^{n\times n}$ is usually referred to as Onsager matrix.
In the notations above, 
the (total) free energy density would be given by 
\begin{align}
  \label{eq:chfe}
  \tilde f(p,z)=\frac12p^\tT\Gamma p+\Psi(z).  
\end{align}
In \eqref{eq:ch},
a typical choice for the value space $\tilde S\subset\R^n$ is the $(n-1)$-dimensional Gibbs-simplex 
(the entries of $z\in\tilde S$ are non-negative and sum up to one), 
and an adapted Onsager matrix is given by $\widetilde{\M}(z) = \operatorname{diag}(z)-zz^\tT$.
Properties of solutions to \eqref{eq:ch} have been studied by various authors,
and the first rigorous existence result has been given in \cite{EGphysD}.

Although the Cahn-Hilliard systems \eqref{eq:ch} themselves
are in principle amendable to the approach that is carried out below,
we prefer to present our method in the formally more transparent setting
in which the value space is given by the cuboid \eqref{eq:cuboid}, 
and the mobility matrix has the diagonal form \eqref{eq:fullydecoupled}.
The \emph{cross diffusion} between the $n$ species is thus induced
be means of the free energy $\ent$
--- like by choosing $\Gamma$ non-diagonal in \eqref{eq:chfe} above ---
but not by means of  the mobility matrix $\M$.
Structurally, our PDE system \eqref{eq:equation4} lies in between 
the Cahn-Hilliard systems \eqref{eq:ch} 
and the fourth order thin film approximations of multi-layered fluids, see e.g. \cite{MatLau},
which come with a fully decoupled $\M$ containing \emph{linear} mobility functions $\mob_j(z_j)=m_jz_j$.
Analytically, \eqref{eq:equation4} is much closer to \eqref{eq:ch},
since the key difficulty arises from the nonlinear mobility functions in $\M$:
the metric $\W_\M$ is much more difficult to handle than a tensorized $L^2$-Wasserstein distance.

% The focus here is on fully decoupled mobility matrices $\M$ with \emph{nonlinear} mobility functions $\mob_j$ on the diagonal;
% the respective equation \eqref{eq:equation4} with linear mobilities $\mob_j(z_j)=m_jz_j$,
% in which case $\W_\M$ is a tensorized version of the celebrated Wasserstein distance,

Apart from the restriction to fully decoupled mobilities, 
we shall make two further simplifying assumptions,
with the intension 
not to obscure the main conceptual ideas that are involved:
\begin{itemize}
  % \item The components $u_j$ in \eqref{eq:equation4} are coupled via the entropy density $f$
  %   --- like by means of a non-diagonal matrix $Q$ in the example \eqref{eq:ch} above ---
  %   but \emph{not} via the mobility matrix $\M$, which is a fully decoupled, see \eqref{eq:fullydecoupled}.
  %   We expect that the existence result still holds at least for mobility matrices with mildly coupled entries,
  %   but this quite restrictive hypothesis appears to be essential for the specific approach that is carried our here.
\item The entropy density $f$ is supposed to be smooth and strictly convex, see Assumption \ref{def:E} below,
  so unlike in the original Cahn-Hilliard problem, there is no meta-stability involved.
  Consequently, the only stationary solution is a spatially homogeneous state.
  The analytical reason for this restriction is simple:
  since we are working on the whole axis $\R$, 
  a choice like in \eqref{eq:chfe} with $\Psi$ being concave
  leads to a free energy $\ent$ in \eqref{eq:energy} that is unbounded from below.
  On the other hand, we admit densities $f$ that do not necessarily decompose as in \eqref{eq:chfe};
  for instance, the matrix $\Gamma$ is allowed to depend on $z$ in a mild way,
  see the example in Remark \ref{ex:HF} below.  
\item The space dimension is one, which helps to simplify notations significantly. 
  The extension of our results to higher space dimensions would require some more technical effort,
  but no conceptually new ingredients.
\end{itemize}

The proof of existence of weak solutions below 
follows the same main lines as in our recent work \cite{zm2014,lmz2015}.
The backbone is the variational \emph{minimizing movement scheme} for gradient flows \cite{giorgi1983,jko1998,savare2008}: 
For a given step size $\tau>0$, define a sequence $(u_\tau^k)_{k\ge 0}$ recursively by $u_\tau^0:=u^0$,
\begin{align}\label{eq:mms}
  u_\tau^k:=\argmin_{u\in\measm}\left(\frac1{2\tau}\W_\M(u,u_\tau^{k-1})^2+\ent(u)\right)\qquad\text{for }k\in\N.
\end{align}
We show that the piecewise constant (in time) interpolants $u_\tau$ along the sequence $(u_\tau^k)_{k\in\N}$ from \eqref{eq:mms} 
converge in a suitable space as $\tau\searrow 0$ to a limit map $u$ 
which is a weak solution to system \eqref{eq:equation4} in the sense specified in Definition \ref{def:weaksol} below. 
The presented strategy of proof counts by now as ``classical'' in the context of gradient flows in the $L^2$-Wasserstein distance: 
after the seminal contributions by Jordan, Kinderlehrer and Otto \cite{jko1998} and Otto \cite{otto2001} on second order equations, 
it has also been employed e.g. for fourth-order equations \cite{giacomelli2001,gianazza2009,matthes2009,lmz2015},
equations of fractional order \cite{liso}, 
and systems \cite{laurencot2011,bcc2012,blanchet2012,zinsl2012,blanchet2014,kmx2015,zinsl2015}. 
For the distances with nonlinear mobility, existence results have been derived by this method 
e.g. in \cite{lisini2012,zinsl2016,zinsl2016b} for the scalar case.
% (for results in abstract metric spaces, see the monograph \cite{savare2008} or the recent articles \cite{ferreira2015,plazotta2016} which extend the theory to the nonautonomous framework). 

% Concerning alternative current approaches to genuine cross-diffusion systems of the form \eqref{eq:pdesystem},
% we mention the recent existence results from \cite{jz2015,jz2016}, which use the boundedness-by-entropy principle from \cite{juengel2014}.
% In comparison to these and classical results (see, e.g., \cite{lady1968,amann}), 
% we obtain solutions of lower regularity, but can allow for more general initial data.

As a closing remark in favor of the variational approach presented here
in comparison to the original existence proof for \eqref{eq:ch} from \cite{EGphysD}
(or other ``classical'' approaches on the basis of \cite{lady1968,amann} etc.),
we stress that once the gradient flow structure of \eqref{eq:equation4} has been properly identified, 
and the details for the metric $\W_\M$ have been worked out \cite{zm2014},
our construction of weak solutions follows a pre-defined route that works in a rather general setting.
In particular,
there is no need to introduce artificial regularizations of the degeneracies in \eqref{eq:equation4},
the natural hypotheses on the initial datum are obtained from finiteness of the initial free energy,
and --- most importantly --- the key a priori estimate follows naturally 
from the energy dissipation along the heat flow.

\section{Notations, hypotheses and results}
\subsection{Notations}
We recall that $S^\ell,S^r\in\R^n$ are two given vectors with $S^\ell_j<S^r_r$ for all $j\in\{1,\ldots,n\}$ 
that define the value space $S$ as the corresponding $n$-cuboid \eqref{eq:cuboid}. 
As in \cite{zm2014}, we distinguish two qualitatively different cases by introducing a \emph{reference state} $\zref\in S$ 
relatively to which certain quantities (e.g. the mass of an element in $\measm$) are measured:
\begin{enumerate}[{Case }\upshape (A){:}]
\item Reference state $\zref=S^\ell$.
\item Reference state $\zref\in\inn{S}$.
\end{enumerate}
The respective case will always be indicated with (A) and/or (B) if necessary. 
Note that in case (A), the function $u-\zref$ is nonnegative for each $u\in\measm$.

The space of measurable $S$-valued functions $u:\R\to S$ is denoted by $\measm$.
For each $u\in\measm$ and a given reference point $\zref\in S$, 
we introduce the quantities $\mss{u-\zref}\in\R^n$ and $\mom{u-\zref}\in\R^n$ by
\[ 
\left[\mss{u-\zref}\right]_j=\int_\R\big(u_j(x)-\zref_j\big)\dd x, \quad 
\left[\mom{u-\zref}\right]_j=\int_\R x^2\big(u_j(x)-\zref_j\big)\dd x. 
\]
We write $0<\mss{u-\zref}<\infty$ to say that all components are positive and finite,
and we abbreviate $|\mom{u-\zref}|=\sum_{j=1}^n[\mom{u-\zref}]_j$.
On the subspaces $L^2(\R;\R^n)$ and $H^1(\R;\R^n)$ of $\measm$,
we use the following norms:
\begin{align*}
  \|u-\zref\|_{L^2} &= \left(\sum_{j=1}^n\int_\R\big(u_j(x)-\zref_j\big)^2\dd x\right)^{\frac12}, \\
  \|u-\zref\|_{H^1} &= \left(\|u-\zref\|_{L^2}^2+\|\partial_x(u-\zref)\|_{L^2}^2\right)^{\frac12}.
\end{align*}
Now define, in the respective case (A) or (B), the \emph{auxiliary space} $\Xaux\subset\measm$ as follows.
\begin{enumerate}[\upshape (A)]
\item Let $u^0\in\measm$ be the initial datum with $0<\mss{u^0-\zref}<\infty$.
  Set
  \begin{align*}
    \Xaux :=\{u\in\measm:&~\mss{u-\zref}=\mss{u^0-\zref},\\&|\mom{u-\zref}|<\infty,\,\|\partial_x u\|_{L^2}<\infty\}.      
  \end{align*}
\item Here, simply put 
  $\Xaux :=\{u\in\measm:\,\|u-\zref\|_{H^1}<\infty\}$.
\end{enumerate}

\subsection{Hypotheses on the mobility matrix}
For the mobility matrix $\M$, 
we assume that it is \emph{induced} by a function $h:S\to\R$ of the form $h(z)=\sum_{j=1}^n h_j(z_j)$,
that is $(\nabla_z^2 h(z))^{-1}=\M(z)$, or simply $\mob_j=\frac1{h_j''}$, for $z\in\inn{S}$. 
We additionally require that for all $j=1,\ldots,n$, the following holds (compare to \cite{zm2014}):
\begin{enumerate}[(H1)]
\item $h_j$ is $\alpha$-H\"older continuous on $[S^\ell_j,S^r_j]$ for each $\alpha<1$, and smooth on $(S^\ell_j,S^r_j)$;
\item $h_j$ is strictly convex and non-positive with $h_j(S^\ell_j)=h_j(S^r_j)=0$;
  % \item $\lim\limits_{s\searrow S^\ell_j}h_j''(s)=+\infty=\lim\limits_{s\nearrow S^r_j}h_j''(s)$;
\item $\frac1{h_j''}\upharpoonright_{(S^\ell_j,S^r_j)}$ extends continuously to 
  a concave function $\mob_j\in C^2([S^\ell_j,S^r_j])$ with $\mob_j(S^\ell_j)=\mob_j(S^r_j)=0$.
\end{enumerate}
These conditions are rather restrictive: 
essentially, they imply that, as $s\searrow S^\ell_j$ and $z_j\nearrow S^r_j$, respectively;
$h_j$ behaves like a multiple of $(z_j-S^\ell_j)\log(z_j-S^\ell_j)$ and of $(S^r_j-z_j)\log(S^r_j-z_j)$,
and accordingly $\mob_j$ behaves like a multiple of $z_j-S^\ell_j$ and of $S^r_j-z_j$;
see Remark \ref{ex:HF} below.

With $h$ satisfying the conditions above, 
we introduce the associated \emph{heat entropy} functional, 
that will be the source for our key a priori estimate on solutions to \eqref{eq:equation4}.
\begin{definition}[Heat entropy]\label{def:H}
  Let $S$ and $\zref$ as above, and assume that $h$ satisfies (H1)--(H3). 
  Define the \emph{heat entropy} functional $\calH:\measm\to\R_\infty$ by
  \begin{align*}
    \calH(u)=\int_\R h_{\zref}(u)\dd x,
  \end{align*}
  where, depending on the cases (A) or (B),
  \begin{enumerate}[\upshape (A)]
  \item $h_\zref(z):=h(z)$;
  \item $h_\zref(z):=h(z)-h(\zref)-\nabla_z h(\zref)^\tT(z-\zref)$.
  \end{enumerate}
\end{definition}
Note that $h_\zref$ is nonpositive in case (A) and nonnegative in case (B).
Further, in both cases, $h_\zref(\zref)=0$ and $h_\zref$ is strictly convex with $\hess h_\zref=\hess h$.

\subsection{Hypotheses on the free energy}
For the free energy functional $\ent$, we assume the following.
\begin{assumption}[Free energy]\label{def:E}
  We assume that a smooth function $f:\R^n\times S\to\R$ with the following properties is given:
  \begin{enumerate}[(i)]
  \item \emph{Normalization:}
    $f(0,\zref)=0$ and $\nabla_pf(0,\zref)=\nabla_zf(0,\zref)=0$
  \item \emph{Growth condition:}
    there exist $\bar C_f>0$ and $\underline{C}_f>0$ such that for all $p\in\R^n$, $z\in S$ and $(\pi,\zeta)\in\R^n\times\R^n$, 
    one has --- in the respective case (A) or (B) ---
    \begin{enumerate}[(A)]
    \item $\quad\underline{C}_f |\pi|^2\le \begin{pmatrix}\pi\\\zeta\end{pmatrix}^\tT\nabla^2_{(p,z)}f(p,z)\begin{pmatrix}\pi\\\zeta\end{pmatrix}\le \bar C_f(|\pi|^2+|\zeta|^2)$;
    \item $\quad\underline{C}_f (|\pi|^2+|\zeta|^2)\le \begin{pmatrix}\pi\\\zeta\end{pmatrix}^\tT\nabla^2_{(p,z)}f(p,z)\begin{pmatrix}\pi\\\zeta\end{pmatrix}\le \bar C_f(|\pi|^2+|\zeta|^2)$.
    \end{enumerate}
  \end{enumerate}
  The \emph{free energy} functional $\ent:\measm\to\R_\infty$ is then defined by
  \begin{align*}
    \ent(u)=\begin{cases}
      \int_\R f(\partial_x,u)\dd x&\text{if }u \in \Xaux, \\ 
      +\infty&\text{otherwise.}
    \end{cases}
  \end{align*}
\end{assumption}
The assumption of normalization has been made for convenience,
in order to avoid additional terms inside the integral representation of $\ent$.
Note that any smooth free energy density $g:\R^n\times S\to\R$ satisfying the growth conditions
gives rise to a normalized $f$ that satisfies the growth conditions 
with the same constants $\bar C_f>0$ and $\underline{C}_f>0$
via
\[ f(p,z) := g(p,z) - \big[g(0,\zref) + p^\tT\nabla_pg(0,\zref) + (z-\zref)^\tT\nabla_zg(0,\zref)\big]. \]
\begin{remark}[Examples]\label{ex:HF}
  \begin{enumerate}[(a)]
  \item The paradigmatic example for $h$ satisfying \break \textup{(H1)--(H3)} is given by
    \begin{align*}
      h_j(s)=\begin{cases}
        (s-S^\ell_j)\log(s-S^\ell_j)+(S^r_j-s)\log\left(S^r_j-s\right)&\\
        \quad-(S^r_j-S^\ell_j)\log(S^r_j-S^\ell_j)&\text{if }s\in(S^\ell_j,S^r_j),\\ 
        0&\text{if }s\in \{S^\ell_j,S^r_j\},
      \end{cases}
    \end{align*}
    and yields
    \begin{align*}
      \mob_j(s)=\frac1{S^r_j-S^\ell_j}(s-S^\ell_j)(S^r_j-s),
    \end{align*}
    for each $j\in\{1,\ldots,n\}$.
  \item The Cahn-Hilliard functional given in \eqref{eq:chfe} is admissible,
    provided that the homogeneous free energy density $\Psi$ is uniformly convex, 
    i.e., $\nabla_z^2\Psi\ge\lambda$ for some $\lambda>0$.
    An admissible generalization of that example is
    \begin{align*}
      f(p,z)=\frac12 \big(\varepsilon+a(z)\big)\,p^\tT\Gamma p+\Psi(z),
    \end{align*}
    where $\Gamma\in\Matn$ is still a symmetric positive definite matrix,
    $\varepsilon>0$ is a constant,
    $\Psi:S\to\Rp$ is smooth and uniformly convex, jointly in the $n$ coordinates of $z$,
    and $a:S\to\Rp$ is such that
    \[ a(z)\dff^2a(z)\ge\dff a(z)^\tT\dff a(z) \]
    for all $z\in S$.
    For instance, one may choose
    \[ a(z) = \sum_{k=1}^n e^{\mu_kz_k} \]
    with arbitrary exponential weights $\mu_1,\ldots,\mu_n\in\R$.
  \end{enumerate}
\end{remark}

\subsection{Summary of results}
Our notion of weak solution is as follows:
\begin{definition}[Weak solution]\label{def:weaksol}
  A map $u:\R_{\ge 0}\times\R\to S$ is called \emph{weak solution} to \eqref{eq:equation4} 
  % with the initial condition $u^0$
  if the following holds:
  \begin{enumerate}[(a)]
  \item $u(t,\cdot)\in\Xaux$ for all $t>0$;
  \item $u-\zref\in L^2_\loc(\R_{\ge 0};H^2(\R;\R^n))$;
    % \item $u(t,\cdot)$ converges narrowly to $u^0$ as $t\searrow0$;
    % $u(0,\cdot)=u^0$ almost everywhere on $\R$;
  \item $u$ satisfies the following integral formulation of \eqref{eq:equation4},
    \begin{align}\label{eq:cweak}
      \begin{split}
        &\int_0^\infty \int_\R \big(\partial_t\varphi^\tT u+\nonl(u)[\varphi]\big)\dd x\dd t=0,\\
        &\text{with}\quad
        \nonl(u)[\varphi]=\partial_x\big(\M(u)\partial_x\varphi\big)^\tT\,
        \left[\nabla_{z}f(\partial_x u, u) -\partial_x\nabla_{p}f(\partial_x u, u)\right],
      \end{split}
    \end{align}
    for all $\varphi\in C^\infty_c(\R_{>0}\times\R;\R^n)$.
  \end{enumerate}
\end{definition}
The existence result that we prove is this:
\begin{thm}[Existence of weak solutions to \eqref{eq:equation4}]
  \label{thm:exist4}
  Assume that $h$, $\Xaux$ and $f$ are subject to the conditions mentioned above. 
  Additionally, suppose that $ u^0\in\measm$ and either, depending on the case,
  \begin{enumerate}
  \item[(A)] $0<\mss{u^0-\zref}<\infty$, $|\mom{ u^0-\zref}|<\infty$, and $\partial_x u^0\in L^2(\R;\R^n)$,
    or
  \item[(B)] $u^0-\zref\in H^1(\R;\R^n)$.
  \end{enumerate}
  Then, there exists a weak solution $u:\R_{\ge 0}\times\R\to S$ with the additional properties
  \begin{align}
    \label{eq:u0}
    &u(0,\cdot)=u^0;\\
    \label{eq:Wcts}
    &u\in C^{1/2}(\R_{\ge 0};(\measm,\W_\M));\\
    \nonumber
    &u-\zref\in L^\infty(\R_{\ge 0};H^1(\R;\R^n));
  \end{align}
  and, in the case (A), for all $t>0$:
  \begin{align}
    \label{eq:exaddon}
    &\mss{u(t)-\zref}=\mss{u^0-\zref};\\
    &|\mom{ u(t)-\zref}|<\infty.
  \end{align}
\end{thm}
We emphasize that the initial condition \eqref{eq:u0} only makes sense 
in combination with the continuity from \eqref{eq:Wcts}.
In particular, $u(t,\cdot)$ converges to $u^0$ weakly$\ast$ as $t\searrow0$.

There is an additional a priori estimate that follows from our construction, 
and that one implies some information on the long time asymptotics of the corresponding weak solutions.
We summarize this:
\begin{thm}[Long time asymptotics]
  \label{thm:long}
  The solutions whose existence has been discussed in Theorem \ref{thm:exist4} above have the following additional properties:
  \begin{itemize}
  \item In case (A): 
    for each $p<\frac4{15}$, there is a $K_p>0$ such that
    \begin{align}
      \label{eq:decay1}
      % \|u(t)-\zref\|_{H^1} 
      \ent(u(T))\le K_pT^{-p} \quad \text{for all $T\ge1$}.
    \end{align}
    In particular, $u-\zref$ converges to zero in $H^1(\R;\R^n)$ as $t\to\infty$.
  \item Given $\delta>0$, there exists 
    \begin{itemize}
    \item in case (A), for each $p>\frac13$ a set $\Theta\subset\R_{>0}$ 
      such that (denoting the Lebesgue measure on $\R$ by $\mathcal{L}$)
      \[\lebm{\Theta\cap\{t\le T\}}\le\delta T^p \quad\text{for every $T>0$};\]
    \item in case (B), a set $\Theta\subset\R_{>0}$ with $\lebm{\Theta}\le\delta$,
    \end{itemize}
    such that $\|\partial_x^2u(t)\|_{L^2}$ converges to zero,
    and $u(t)$ converges uniformly to the constant $\zref$ as $t\to\infty$ with $t\in\R_{>0}\setminus\Theta$.
  \end{itemize}
\end{thm}
These results are probably much weaker than one would expect for the behavior
of the gradient flow for a strictly convex functional like $\ent$.
Note, however, that the uniform convexity of $\ent$ on the flat space $H^1(\R;\R^n)$ is ``not seen'' by the metric $\W_\M$.
Indeed, geodesic $\lambda$-convexity --- even with negative $\lambda$ --- with respect to $\W_\M$
is apparently an extremely rare property for functionals to have \cite{zm2014}.
Therefore, it seems unlikely that significantly better estimates on the long time asymptotics can be derived
without additional information on the solution that go beyond its characterization as gradient flow.

\subsection{Outline of the further paper}
The plan of the paper is the following: 
first, we begin with a summary of the relevant properties of the distance $\W_\M$ and the associated heat entropy $\calH$. 
Then, we study the minimization problem \eqref{eq:mms} and its solutions $u_\tau^k$, 
deriving an additional regularity estimate with the so-called flow interchange technique. 
The same technique is also used to derive an approximate time-discrete weak formulation satisfied by the piecewise-constant interpolants $u_\tau$. 
The proof of Theorem \ref{thm:exist4} is completed by passing to the continuous-time limit $\tau\searrow 0$.
Finally, Theorem \ref{thm:long} follows as a combination of several estimates that have been derived in the course of the existence proof.

\section{Some analytical preliminaries}
In this section, we collect some results on the distance $\W_\M$ and the associated heat entropy $\calH$ which we frequently make use of later. 
We begin with some general theory.
\subsection{Abstract gradient flows}
A sequence $(u_k)_{k\in\N}$ of measurable functions $u_k:\R\to A$ for some closed set $A\subset \R^n$ is said to \emph{converge weakly$\ast$} to some limit map $u:\R\to A$, if for all $\rho\in C^0_c(\R;\R^n)$, one has
\begin{align*}
\lim_{k\to\infty}\int_\R \rho(x)^\tT u_k(x)\dd x&=\int_\R \rho(x)^\tT u(x)\dd x.
\end{align*}
We will use the following notion of gradient flow from \cite{daneri2008,savare2008}:
\begin{definition}[$\kappa$-flows]
Let $\auxil:\X\to\R_\infty$ be a proper and lower semicontinuous functional on the (pseudo-)metric space $(\X,\metr)$ and let $\kappa\in\R$. A continuous semigroup $\flow^{\auxil}$ on $(\X,\metr)$ satisfying
the \emph{evolution variational estimate}
 \begin{align*}
   \frac{1}{2}\frac{\dd^+}{\dd s}\metr^2(\flow_{s}^{\auxil}(w),\tilde w)+\frac{\kappa}{2}\metr^2(\flow_{s}^{\auxil}(w),\tilde w)+\auxil(\flow_{s}^{\auxil}(w))
   &\le\auxil(\tilde w)
\end{align*}
for arbitrary $w,\tilde w\in\operatorname{Dom}(\mathcal{A})$ and for all $s\ge 0$, as well as the monotonicity condition
\begin{align*}
\auxil(\flow_t^\auxil(w))\le \auxil(\flow_s^\auxil(w))\quad\forall 0\le s\le t
\end{align*}
for all $w\in\X$,
is called \emph{$\kappa$-flow} or \emph{gradient flow} of $\auxil$.
\end{definition}
\begin{remark}
  The concept of $\kappa$-contractive flows is closely related to that of geodesic $\kappa$-convexity.
  Indeed, if a functional $\auxil:\X\to\R_\infty$ induces a $\kappa$-flow, it also is $\kappa$-convex along geodesics \cite{daneri2008,liero2013}.
  The converse is more subtle, and typically requires additional property of $\auxil$.
  In the context of the $L^2$-Wasserstein metric, this has been exhaustively discussed, see e.g. \cite{mccann1997,otto2005}.
\end{remark}
The above-mentioned notion of contractive gradient flow is useful for the derivation of a priori estimates via entropy dissipation. 
Specifically, we use the \emph{flow interchange technique} from \cite{matthes2009}:
\begin{thm}[Flow interchange lemma {\cite[Thm. 3.2]{matthes2009}}]\label{thm:flowinterchange}
Let $\mathcal{B}$ be a proper, lower semicontinuous and geodesically $\lambda$-convex functional on $(\X,\metr)$ and assume that there exists a $\lambda$-flow $\flow^\mathcal{B}$. Let furthermore $\mathcal{A}$ be another proper, lower semicontinuous functional on $(\X,\metr)$ such that $\operatorname{Dom}(\mathcal{A})\subset \operatorname{Dom}(\mathcal{B})$. Assume that, for arbitrary $\tau>0$ and $\tilde w\in \X$, the Yosida penalized functional $w\mapsto \frac1{2\tau}\metr^2(w,\tilde w)+\auxil(w)$ possesses a minimizer $w^*$ on $\X$.

Then, the following holds:
\begin{align*}
\mathcal{B}(w^*)+\tau \mathrm{D}^\mathcal{B}\mathcal{A}(w^*)+\frac{\lambda}{2}\metr^2(w^*,\tilde w)&\le \mathcal{B}(\tilde w).
\end{align*}
There, $\mathrm{D}^\mathcal{B}\mathcal{A}(w)$ denotes the \emph{dissipation} of the functional $\mathcal{A}$ along the $\lambda$-flow $\flow^{\mathcal{B}}$ of the functional $\mathcal{B}$ with starting point $w$, i.e.
\begin{align*}
\mathrm{D}^\mathcal{B}\mathcal{A}(w):=\limsup_{s\searrow 0}\frac{\mathcal{A}(w)-\mathcal{A}(\flow_{s}^{\mathcal{B}}(w)    )   }{s}.
\end{align*}
\end{thm}
We need another auxiliary theorem: since system \eqref{eq:equation4} is nonlinear, only weak convergence of the family of discrete solutions $(u_\tau)_{\tau>0}$ does not suffice to obtain a weak solution in the limit as $\tau\searrow 0$. However, in this metric framework, the following theorem provides an approach similarly to the classical Aubin-Lions compactness lemma:
\begin{thm}[Extension of the Aubin-Lions lemma {\cite[Thm. 2]{rossi2003}}]\label{thm:ex_aub}
Let $\ban$ be a Banach space and let $\mathcal{A}:\,\ban\to[0,\infty]$ be lower semicontinuous and have relatively compact sublevels in $\ban$. Let furthermore $\W:\,\ban\times \ban\to[0,\infty]$ be lower semicontinuous and such that $\W(w,\tilde w)=0$ for $w,\tilde w\in \dom(\mathcal{A})$ implies $w=\tilde w$.

If for a sequence $(w_k)_{k\in\N}$ of measurable functions $w_k:\,(0,T)\to \ban$, one has
\begin{align}
\sup_{k\in\N}\int_0^T\mathcal{A}(w_k(t))\dd t&<\infty,\label{eq:hypo1}\\
\lim_{h\searrow 0}\sup_{k\in\N}\int_0^{T-h}\W(w_k(t+h),w_k(t))\dd t&=0,\label{eq:hypo2}
\end{align}
then there exists a subsequence that converges in measure w.r.t. $t\in(0,T)$ to a limit $w:\,(0,T)\to \ban$.
\end{thm}

\subsection{Modified Wasserstein distances}
\label{sct:WM}
The pseudo-metric $\W_\M$ on $\measm$ is defined via an extension of the classical Benamou-Brenier formula \cite{brenier2000} for the $L^2$-Wasserstein distance: 
for $u_0,u_1\in\measm$, let
\begin{align}\label{eq:defmetr}
  \begin{split}
    \W_\M(u_0,u_1)&:=\inf\bigg\{\bigg. \int_0^1\int_\R w_s^\tT \M(u_s)^{-1}w_s\dd x\dd s:\\&\qquad(u_s,w_s)_{s\in[0,1]}\in\mathscr{C},~u_s|_{s=0}=u_0,~u_s|_{s=1}=u_1\bigg.\bigg\}^{1/2},
  \end{split}
\end{align}
where $\mathscr{C}$ is a suitable subclass of solutions to the (multi-component) \emph{continuity equation} $\partial_s u_s+\partial_x w_s=0$ on $[0,1]\times\R$ in the sense of distributions, see \cite{dns2009,zm2014} for more details. Since the $n$ species $u_1,\ldots,u_n$ do not interact with each other via $\M$ due to its decoupled structure \eqref{eq:fullydecoupled},
one has that
\begin{align*}
  \W_\M(u,\tilde u)^2=\W_{\mob_1}(u_1,\tilde u_1)^2+\cdots+\W_{\mob_n}( u_n,\tilde u_n)^2,
\end{align*}
i.e., $\W_\M$ is the canonical product of the distances $\W_{\mob_k}$ for each of the scalar components which have been defined in \cite{dns2009,lisini2010} as in \eqref{eq:defmetr} for the scalar case $n=1$.

The following two theorems are collections of results from \cite{dns2009,lisini2010,lisini2012,zm2014}.
\begin{thm}[Properties of the distance $\W_\M$]\label{thm:propW}
  Assume that $h:S\to\R$ of the form $h(z)=\sum_{j=1}^n h_j(z_j)$ satisfies \textup{(H1)--(H3)} and that the fully decoupled mobility $\M$ is induced by $h$. The following statements hold:
  \begin{enumerate}[(a)]
  \item For all $j\in\{1,\ldots,n\}$, one has $m_j\in C^2([S^\ell_j,S^r_j])$ with $m''_j(s)\le 0$ for all $s\in(S^\ell_j,S^r_j)$, $m_j(s)>0$ for all $s\in(S^\ell_j,S^r_j)$ and $m_j(S^\ell_j)=0=m_j(S^r_j)$. 
    In consequence, the minimization problem \eqref{eq:defmetr} in the definition of $\W_\M$ on $\measm$ is convex.
  \item If $u\in\measm$ and $K\subset\measm$ are such that $\W_\M(\cdot,u)$ is bounded on $K$, then $K$ is relatively compact w.r.t. weak$\ast$-convergence.
  \item The functional $\W_\M$ is in both arguments lower semicontinuous w.r.t. \break weak$\ast$-convergence.
  \item There exists a constant $L>0$ such that for all $u_0,u_1\in\measm$, one has
    \begin{align*}
      |\mom{u_0-S^\ell}|&\le L\left(|\mom{u_1-S^\ell}|+\W_\M(u_0,u_1)^2\right).
    \end{align*}
  \end{enumerate}
  As a consequence of the above, 
  the subset of all $u\in\measm$ which have finite $\W_\M$-distance to a specified reference $u_*\in\measm$
  is a complete metric space with respect to $\W_\M$,
  and the convergence w.r.t. $\W_\M$ is stronger than weak$\ast$ convergence on that set.
\end{thm}
Even though geodesic convexity w.r.t. $\W_\M$ is a very rare property \cite{carrillo2010,zm2014}, the following holds:
\begin{thm}[Heat entropy and potential energy]\label{thm:HV}
Assume that $h:S\to\R$ of the form $h(z)=\sum_{j=1}^n h_j(z_j)$ satisfies \textup{(H1)--(H3)} and that the fully decoupled mobility $\M$ is induced by $h$. The following statements hold:
\begin{enumerate}[(a)]
\item The heat entropy $\calH$ defined in Definition \ref{def:H} is finite on $\Xaux$.
  More precisely, for each $\alpha\in(\frac13,1)$, there exists a constant $C_\alpha>0$ such that, 
  depending on the case \textup{(A)} or \textup{(B)},
  \begin{align}
    &\textup{(A)}\qquad\qquad\qquad -C_\alpha\big(|\mss{u-\zref}|+|\mom{u-\zref}|^\alpha\big)\le\calH(u)\le 0,\label{eq:Hmom}\\
    &\textup{(B)}\qquad\qquad\qquad 0\le\calH(u)\le C_\alpha\|u-\zref\|_{L^2}^2,\label{eq:HL2}
  \end{align}
  for all $u\in\Xaux$. 
  Furthermore, 
  $\calH$ is $0$-convex along geodesics in the space $(\measm,\W_\M)$, and it induces a $0$-flow $\flow^\calH$ which coincides with the (multi-component) heat flow, i.e.,
  \begin{align*}
    \partial_s \flow^\calH_s(u)=\partial_{x}^2\flow^\calH_s(u)\quad\text{for all }s>0,\text{ and } \flow^\calH_0(u)=u.
  \end{align*}
\item Fix $\alpha>0$ and $\rho\in C^\infty_c(\R;\R^n)$, 
  and define the (regularized) \emph{potential energy} functional
  \begin{align*}
    \mathcal{V}:\measm\to \R_\infty,\quad \mathcal{V}(u)=\alpha \calH(u)+\int_\R\rho(x)^\tT u(x)\dd x.
  \end{align*}
  There exists $C>0$ depending only on $\M$ and $\rho$ 
  such that $\mathcal{V}$ is $\lambda$-convex along geodesics in $(\measm,\W_\M)$ 
  and induces a $\lambda$-flow $\flow^\mathcal{V}$, for $\lambda:=-C(\frac1{\alpha}+1)$. 
  Moreover, if $\hat u\in\measm$ is such that $\hat u-\zref\in H^1(\R)$, 
  then $u_s:=\flow^{\mathcal{V}}_s(\hat u)$ is a classical solution to the viscous nonlinear continuity equation
  \begin{align}
    \label{eq:viscous}
    \partial_su_s=\alpha\partial_{x}^2u_s+\partial_x(\M(u_s)\partial_x\rho), \quad\text{for all $s>0$},
  \end{align}
  the curve $s\mapsto u_s-\zref$ is continuous in $H^1(\R)$ on $\R_{\ge0}$,
  and $u_s(x)$ is smooth with respect to $(s,x)\in\R_{>0}\times\R$.
\end{enumerate}
\end{thm}
\begin{proof}
  These claims have been shown in \cite{zm2014} --- see in particular Proposition 5.8 therein ---
  except for the refined version of the lower bound in \eqref{eq:Hmom}, which we prove now.
  Since each $h_j$ is $\alpha$-H\"older continuous for any given $\alpha\in(\frac13,1)$ by hypothesis (H1),
  there is some $K_\alpha$ such that $h_j(z_j)\ge-K_\alpha(z_j-\zref_j)^\alpha$.
  Hence
  \begin{align*}
    &\int_\R h_k(u_k(x))\dd x 
    \ge -K_\alpha\int_\R \big(u_k(x)-\zref_k\big)^\alpha\dd x \\
    & \ge -K_\alpha\left(\int_\R(1+x^2)^{-\frac\alpha{1-\alpha}}\dd x\right)^{1-\alpha}\left(\int_\R(1+x^2) \big(u_k(x)-\zref_k\big)\dd x\right)^\alpha \\
    &\ge -K_\alpha\omega_\alpha\big(\big[\mss{u-\zref}\big]_k^\alpha+\big[\mom{u-\zref}\big]_k^\alpha\big),
  \end{align*}
  where $\omega_\alpha=\int_\R(1+x^2)^{-\frac\alpha{1-\alpha}}\dd x$ is finite since $\alpha\in(\frac13,1)$.
  Addition of the estimates above for $k=1,\ldots,n$ and another elementary estimate on the $\alpha$ power yield \eqref{eq:Hmom}.
\end{proof}

\section{The variational scheme}
In this section, we study the variational scheme \eqref{eq:mms} and show that the discrete solution $u_\tau$ is well-defined and enjoys an additional regularity property. For the latter, we use the flow interchange lemma (Theorem \ref{thm:flowinterchange}) with the heat entropy $\calH$ from Definition \ref{def:H} as auxiliary functional.
In advance, we prove some elementary properties of the free energy functional $\ent$.
\begin{prop}[Properties of the free energy]\label{prop:propE}
The following statements hold:
\begin{enumerate}[(a)]
\item For all $u\in\Xaux$, the following holds depending on the case \textup{(A)} or \textup{(B)}:
\begin{align*}
&\textup{(A)}\qquad\qquad\qquad\underline{C}_f\|\partial_x(u-\zref)\|_{L^2}^2\le \ent(u)\le \overline{C}_f\|u-\zref\|_{H^1}^2;\\
&\textup{(B)}\qquad\qquad\qquad\underline{C}_f\|u-\zref\|_{H^1}^2\le \ent(u)\le \overline{C}_f\|u-\zref\|_{H^1}^2.
\end{align*}
In particular, $\ent$ is finite on $\Xaux$.
\item If $u_k-\zref\rightharpoonup u-\zref$ weakly in $H^1(\R;\R^n)$, then
\begin{align*}
\ent(u)&\le \liminf_{k\to\infty} \ent(u_k).
\end{align*}
\end{enumerate}
\end{prop}
\begin{proof}
Notice for (a) that the indicated definition of $\ent$ yield together with the convexity and growth properties from Assumption \ref{def:E} and Taylor's theorem that
\begin{align*}
\underline{C}_f\int_\R |\partial_x (u-\zref)|^2\dd x&\le \ent(u)\le \bar{C}_f \int_\R \left[|\partial_x (u-\zref)|^2+|u-\zref|^2\right]\dd x,
\end{align*}
in case (A), or
\begin{align*}
\underline{C}_f\int_\R \left[|\partial_x (u-\zref)|^2+|u-\zref|^2\right]\dd x&\le \ent(u)\le \bar{C}_f \int_\R \left[|\partial_x (u-\zref)|^2+|u-\zref|^2\right]\dd x,
\end{align*}
in case (B), respectively. The lower semicontinuity in (b) of $\ent$ with respect to weak convergence in $H^1(\R;\R^n)$ is a consequence of convexity and nonnegativity of the integrand in $\ent$, see for instance \cite[Thm. 10.16]{renardy2004}.
\end{proof}
In advance of the proof of well-posedness of the variational scheme \eqref{eq:mms}, we make the following observation in the case (A): Since $\|u-\zref\|_{L^1}$ is fixed on $\Xaux$ and $\|u-\zref\|_{L^\infty}$ is uniformly bounded for $u\in\measm$ as the value space $S$ is compact, an interpolation inequality shows that all $L^p(\R;\R^n)$ norms of elements in $\Xaux$ are bounded by a uniform constant. Thus, in view of the assumed convexity of $f$ from Assumption \ref{def:E}, it suffices to control the gradient in order to control the whole Sobolev norm.
\begin{prop}[Minimizing movement scheme]\label{prop:minmovds4}
Let $\tau>0$ and $\tilde u\in \Xaux$. Then, the minimization problem in \eqref{eq:mms} for $u_\tau^{n-1}=\tilde u$ has a solution $u^*\in \Xaux$. Moreover, one has
\vspace{-\topsep}
\begin{align}
\label{eq:addregds4}
\tau \|\partial_{x}^2 u^*\|^2_{L^2}&\le \frac1{\underline{C}_f}[\calH(\tilde u)-\calH( u^*)],
\end{align}
where the heat entropy $\calH$ is defined as in Definition \ref{def:H}. In particular, $u^*-\zref\in H^2(\R;\R^n)$.
\end{prop}
\begin{proof}
We prove the existence of a minimizer in \eqref{eq:mms} with the direct method from the calculus of variations. Thanks to the coercivity estimates from Proposition \ref{prop:propE}(a), the Yosida penalized functional
\begin{align*}
\ent_\tau(\cdot;\tilde u):\,\measm\to\R_\infty,\quad u\mapsto \frac1{2\tau}\W_\M^2(u,\tilde u)+\ent(u)
\end{align*}
is nonnegative. Consequently, a minimizing sequence $(u_k)_{k\in\N}$ for $\ent_\tau(\cdot;\tilde u)$ necessarily belongs to $\Xaux$ and satisfies for all $k\in\N$:
\begin{align*}
\W_\M^2(u_k,\tilde u)&\le C,\\
\|u_k-\zref\|_{H^1}&\le C,
\end{align*}
for some constant $C>0$ which does not depend on $k$. The relative compactness property in Theorem \ref{thm:propW}(b) and Alaoglu's theorem then yield the existence of a map $u^*\in\Xaux$ such that, on a suitable subsequence, $u_k\stackrel{\ast}{\rightharpoonup} u$ in $\measm$ and $u_k-\zref\rightharpoonup u^*-\zref$ weakly in $H^1(\R;\R^n)$. With the lower semicontinuity from Theorem \ref{thm:propW}(c) and Proposition \ref{prop:propE}(b), we infer that $u^*$ indeed is a minimizer of $\ent_\tau(\cdot;\tilde u)$.

For the proof of the additional regularity estimate \eqref{eq:addregds4}, we use the flow interchange principle for the auxiliary functional $\calH$ which induces the heat flow as $0$-flow, recall Theorem \ref{thm:HV}(a). Theorem \ref{thm:flowinterchange} now yields
\begin{align}\label{eq:fiH}
\tau\dff^\calH\ent(u^\ast)&\le\calH(\tilde u)-\calH(u^\ast).
\end{align}
We now derive the dissipation $\dff^\calH\ent(u^\ast)$, and write $u_s:=\flow_s^\calH(u^\ast)$ for brevity:
\begin{align*}
  -\frac{\dd}{\dd s}\ent(u_s)
  &=\int_\R\big(\big.\partial_x u_s^\tT\nabla^2_{z}f(\partial_x u_s, u_s)\partial_x u_s-\partial_x u_s^\tT\partial_x[\nabla^2_{p}f(\partial_x u_s, u_s)\partial_{x}^2 u_s]\\&\qquad-\partial_x u_s^\tT\partial_x[\nabla_{pz}f(\partial_x u_s, u_s)]\partial_x u_s\big.\big)\dd x\\
  &=\int_\R \begin{pmatrix}\partial_{x}^2 u_s\\\partial_x u_s\end{pmatrix}^\tT\nabla^2_{(p,z)}f(\partial_x u_s, u_s)\begin{pmatrix}\partial_{x}^2 u_s\\\partial_x u_s\end{pmatrix}\dd x\ge \underline{C}_f\|\partial_{x}^2 u_s\|_{L^2}^2,
\end{align*}
using integration by parts and in the last step Assumption \ref{def:E}. Weak lower semicontinuity of the $L^2$-norm yields
\begin{align*}
\dff^\calH\ent(u^\ast)&\ge \liminf_{s\searrow 0}\left(-\frac{\dd}{\dd s}\ent(u_s)\right)\ge \underline{C}_f\|\partial_{x}^2 u^\ast\|_{L^2}^2,
\end{align*}
and hence \eqref{eq:addregds4} by combining with \eqref{eq:fiH}.
\end{proof}
Hence, the minimizing movement scheme \eqref{eq:mms} produces, 
for each initial datum $u^0\in\Xaux$ and $\tau>0$, a sequence $(u_\tau^k)_{k\in\N}$ and a time-discrete solution $u_\tau$. 
We now prove a series of a priori estimates.
\begin{prop}[Classical energy estimates]
  \label{prp:apriori1}
  For each $\tau>0$, the following holds.
  \begin{enumerate}[(a)]
  \item The function $t\mapsto\ent(u_\tau(t))$ is monotonically decreasing in $t\ge0$.
  \item $\displaystyle{\sum_{k=1}^\infty\W_\M(u_\tau^k,u_\tau^{k-1})^2 \le2\tau(\ent(u^0)-\inf\ent)}$. 
  \item For all $s,t\ge0$, one has
    \begin{align*}
      \W_\M( u_\tau(s), u_\tau(t))\le \left[2(\ent( u^0)-\inf\ent)\max(\tau,|t-s|)\right]^{1/2}.
    \end{align*}
  \item Only in case (A): there exists a $\tau$-independent constant $M>0$ 
    such that 
    $\mom{u_\tau(t)-\zref}\le M(1+t)$ holds for all $t\ge0$.
  \end{enumerate}
\end{prop}
\begin{proof}
  The derivation of the estimates (a), (b) and (c) are indeed classical, see e.g. \cite{savare2008};
  they follow from $\ent_\tau(u_\tau^k;u_\tau^{k-1})\le\ent_\tau(u_\tau^k;u_\tau^k)$ for each $k$.
  To conclude (d), simply combine (c) with Theorem \ref{thm:propW}(d). 
\end{proof}
\begin{lemma}
  \label{lem:H2}
  Under the general hypotheses, we have
  \begin{itemize}
  \item in case (A):
    for each $q>\frac13$, there exists a $C_q>0$ such that
    \begin{align*}
      \int_0^T\|\partial_x^2u_\tau(t)\|_{L^2}^2\dd t \le C_q(1+T^q) \quad\text{for all $T>0$};
    \end{align*}
  \item in case (B):
    there exists a $C>0$ such that
    \begin{align*}
      \int_0^T\|\partial_x^2u_\tau(t)\|_{L^2}^2\dd t \le C\quad\text{for all $T>0$};
    \end{align*}
  \end{itemize}
  and the respective constants $C_q$ and $C$ are $\tau$-uniform.
\end{lemma}
\begin{proof}
  Since $u_\tau$ is piecewise constant in $t$, we obtain:
  \begin{align*}
    \int_0^T\|\partial_x u_\tau(t)\|_{L^2}^2\dd t
    &\le \sum_{k=1}^K\tau\|\partial_x  u_\tau^k\|_{L^2}^2
    \le \sum_{k=1}^K\left[\frac1{\underline{C}_f}(\calH( u_\tau^{k-1})-\calH( u_\tau^k))\right],
  \end{align*}
  for $K:=\gau{\frac{T}{\tau}}+1$, where we used \eqref{eq:addregds4} in the last step. 
  Simplifying the telescopic sum, we get
  \begin{align*}
    \int_0^T\|\partial_x u_\tau(t)\|_{L^2}^2\dd t&\le \frac1{\underline{C}_f}(\calH( u^0)-\calH(u_\tau^K)).
  \end{align*}
  In the case (A), we employ the bound \eqref{eq:Hmom} on $\calH$ by the second moment, 
  and use property (d) from Proposition \ref{prp:apriori1}.
  For the estimate in case (B), we use \eqref{eq:HL2}.
\end{proof}
\begin{prop}[Further a priori estimates]
  \label{prp:apriori2}
  % Let $u^0\in\Xaux$. 
  For given $T>0$ and $\bar\tau>0$, there exist constants $C>0$ such that for all $\tau\in (0,\bar\tau]$, the following holds:
  \begin{enumerate}[(a)]
  \item $\| u_\tau-\zref\|_{L^\infty([0,T];H^1)}\le C$.
  \item $\| u_\tau-\zref\|_{L^2([0,T];H^2)}\le C$.
  \end{enumerate}
\end{prop}
\begin{proof}
  Using Proposition \ref{prp:apriori1}(a) and the coercivity condition on $\ent$ from Proposition \ref{prop:propE}(a), 
  we immediately obtain (a) above (recall that in the case (A), $\|u_\tau(t,\cdot)-\zref\|_{L^2}$ is bounded by a fixed constant). 
  Estimate (b) is a consequence of Lemma \ref{lem:H2}.
\end{proof}

\section{Discrete weak formulation}
This section is concerned with the derivation of the approximate, time-discrete weak formulation satisfied by the discrete solution $u_\tau$. 
One more time, we employ the flow interchange technique: 
this time, the auxiliary flow is induced by the regularized potential energy $\mathcal{V}$ from Theorem \ref{thm:HV}(b). 
We introduce the following notation: 
for a temporal test function $\psi\in C^\infty_c(\R_{>0})$ and each $\tau>0$,
define the piecewise constant approximation $\psi_\tau:\R_{\ge0}\to\R$ by
\begin{align*}
  \psi_\tau(s)&:=\psi\left(\gau{\frac{s}{\tau}}\tau\right) \quad \text{for all $s\ge0$}.
\end{align*}
\begin{lemma}[Discrete weak formulation]
  \label{lemma:dweak}
  Fix a terminal time $T>0$ 
  and a pair of test functions $\rho\in C^\infty_c(\R;\R^n)$, $\psi\in C^\infty_c((0,T))$. 
  Let $\alpha>0$, and set $\lambda=\lambda(\alpha)=-C\left(\frac1{\alpha}+1\right)$ with $C$ from Theorem \ref{thm:HV}(b). 
  Then, the discrete solution $u_\tau$ obtained from the scheme \eqref{eq:mms} 
  satisfies the following \emph{discrete weak formulation}:
  \begin{align}
    &\Bigg|\Bigg.\int_0^\infty \int_\R \bigg[\bigg.\frac{\psi_\tau(t+\tau)-\psi_\tau(t)}{\tau}\rho^\tT u_\tau(t)
    +\psi_\tau(t) \nonl\big(u_\tau(t)\big)[\rho]\bigg.\bigg]\dd x\dd t\Bigg.\Bigg|\nonumber\\
    &\le\|\psi\|_{C^0}\big(\alpha[\calH(u^0)+M(1+T)]+(-\lambda)\tau\ent(u^0)\big),
    \label{eq:dweak}
  \end{align}
  where $M>0$ is a $\tau$-independent constant, 
  and the nonlinear operator $\nonl$ is given in \eqref{eq:cweak}.
\end{lemma}
\begin{proof}
  Since by Theorem \ref{thm:HV}(b) the regularized potential energy $\mathcal{V}$ induces a $\lambda$-flow on $(\measm,\W_\M)$, 
  we may apply the flow interchange lemma (Theorem \ref{thm:flowinterchange}) to obtain for all $k\in\N$:
  \begin{align}
    \label{eq:potflowint}
    \pot( u_\tau^k)+\tau \dff^\pot\ent( u_\tau^k)+\frac{\lambda}{2}\W_\M^2( u_\tau^k, u_\tau^{k-1})&\le \pot( u_\tau^{k-1}).
  \end{align}
  Recall that the flow $u_s:=\flow^\pot_s( u_\tau^k)$ generated by $\mathcal{V}$ 
  satisfies the viscous nonlinear continuity equation \eqref{eq:viscous},
  is such that $u_s(x)$ is smooth with respect to $(s,x)\in\R_{>0}\times\R$,
  and $u_s-\zref\to u_\tau^k-\zref$ in $H^1(\R)$ as $s\searrow0$.
  The corresponding dissipation can thus be calculated explicitly,
  using
  \[ f_s:=f(\partial_xu_s,u_s), \quad \nabla_pf_s:=\nabla_pf(\partial_xu_s,u_s), \quad \text{etc.}, \]
  for brevity:
  \begin{align*}
    -\frac{\dd}{\dd s}\ent(u_s)
    &=-\int_\R\big[\nabla_p f_s \partial_x\partial_su_s+\nabla_z f_s\partial_su_s\big]\dd x \\
    &=\int_\R\big[\alpha\partial_x^2u_s+\partial_x\big(\M(u_s)\partial_x\rho\big)\big]^\tT
    \big[\partial_x\nabla_pf_s-\nabla_zf_s\big]\dd x \\
    &=\alpha\int_\R{\partial_x^2u_s\choose\partial_xu_s}^\tT\nabla_{(p,z)}^2f_s{\partial_x^2u_s\choose\partial_xu_s}\dd x
    -\int_\R \nonl(u_s)[\rho]\dd x.
  \end{align*}
  Thanks to the coercivity estimate from Assumption \ref{def:E},
  we obtain
  \begin{align}
    \label{eq:passtolimithere}
    -\frac{\dd}{\dd s}\ent(u_s)
    \ge\alpha\underline{C}_f\int_\R|\partial_x^2u|^2\dd x-\int_\R \nonl(u_s)[\rho]\dd x.
  \end{align}
  We wish to pass to the limit $s\searrow0$ in order to produce 
  a sensible lower bound on $\dff^\pot\ent( u_\tau^k)$ in \eqref{eq:potflowint} above;
  this step is a little delicate.
  First, let us consider
  \begin{align}
    \label{eq:Zs}
    Z_s := \partial_x\nabla_pf_s-\nabla_zf_s
    = \nabla_{pp}f_s\partial_x^2u_s + \nabla_{pz}f_s\partial_xu_s-\nabla_zf_s.
  \end{align}
  From Assumption \ref{def:E}, we know 
  that the functions $\nabla_{pp}f$ and $\nabla_{pz}f$ are bounded by $\bar C_f$ on $\R^n\times S$, 
  and it easily follows that the function $\nabla_zf$ is globally Lipschitz, with Lipschitz constant $\bar C_f$.
  Therefore,
  \begin{align*}
    \|Z_s\|_{L^2}
    &\le\|\nabla_{pp}f_s\partial_x^2u_s\|_{L^2}+\|\nabla_{pz}f_s\partial_xu_s\|_{L^2}+\|\nabla_zf_s\|_{L^2} \\
    &\le \bar C_f\big(\|\partial_x^2u_s\|_{L^2} + \|\partial_xu_s\|_{L^2} + \|u_s-\bar z\|_{H^1}\big).
  \end{align*}
  Further, thanks to the smoothness of the mobilities $\mob_j$, see Theorem \ref{thm:propW},
  there is an appropriate constant $C_\M$ such that
  \begin{align*}
    \big\|\partial_x\big(\M(u_s)\partial_x\rho\big)\big\|_{L^2}\le C_\M\|\rho\|_{H^2}(1+\|\partial_xu_s\|_{L^2}).
  \end{align*}
  Note that $\rho\in H^2(\R)$ since $\rho$ is smooth and of compact support.
  This yields a rough lower bound on the dissipation:
  \begin{align*}
    &-\frac{\dd}{\dd s}\ent(u_s) \\
    &\ge\alpha\underline{C}_f\|\partial_x^2u\|_{L^2}^2 
    - C_\M\bar C_f\|\rho\|_{H^2}\big(1+\|u_s-\zref\|_{H^1}\big)\big(\|\partial_x^2u_s\|_{L^2}+2\|u_s-\zref\|_{H^1}\big)\\
    &\ge \frac12\alpha\underline{C}_f\|\partial_x^2u\|_{L^2}^2 - K(1+\|u_s-\zref\|_{H^1}^2),
  \end{align*}
  for some constant $K$ depending on $\alpha$, $\underline C_f$, $\bar C_f$, $C_\M$, and $\|\rho\|_{H^2}$.
  From here, 
  the continuity of $s\mapsto u_s-\zref$ in $H^1(\R)$, 
  and the flow interchange estimate \eqref{eq:potflowint} 
  imply that $\|\partial_x^2u\|_{L^2}$ remains bounded as $s\searrow0$.
  Therefore, in addition to the strong convergence $u_s-\zref\to u_\tau^k-\zref$ in $H^1(\R)$,
  we have the respective convergence also weakly in $H^2(\R)$, and strongly in $W^{1,\infty}(\R)$.

  This is sufficient to pass to the limit $s\searrow 0$ in \eqref{eq:passtolimithere}.
  Using the aforementioned strong convergence in $W^{1,\infty}$,
  it follows that
  \begin{align*}
    \nabla_{pp}f_s=\nabla_{pp}f(\partial_xu_s,u_s) \to \nabla_{pp}f(\partial_xu_\tau^k,u_\tau^k) \quad \text{uniformly on $\R$},
  \end{align*}
  and likewise for $\nabla_{pz}f_s$ and $\nabla_zf_s$.
  For $Z_s$ from \eqref{eq:Zs} this yields, in combination with the weak convergence of $u_s-\zref$ in $H^2(\R)$,
  the weak convergence
  \begin{align*}
    Z_s \rightharpoonup 
    &\nabla_{pp}f(\partial_xu_\tau^k,u_\tau^k)\partial_x^2u + \nabla_{pz}f(\partial_xu_\tau^k,u_\tau^k)\partial_xu-\nabla_zf(\partial_xu_\tau^k,u_\tau^k)\\
    & =\partial_x\nabla_pf(\partial_xu_\tau^k,u_\tau^k)-\nabla_zf(\partial_xu_\tau^k,u_\tau^k)
  \end{align*}
  in $L^2(\R)$.
  Finally, recall that $\rho$ is smooth and of compact support,
  therefore the strong convergence of $u_s-\zref$ in $H^1(\R)$ is more than sufficient to conclude that
  \begin{align*}
    \partial_x\big(\M(u_s)\partial_x\rho\big) \to  \partial_x\big(\M(u_\tau^k)\partial_x\rho\big)
  \end{align*}
  in $L^2(\R)$.
  Hence, using the lower semi-continuity of the $L^2$-norm under weak convergence,
  we end up with
  \begin{align*}
    \dff^\pot\ent( u_\tau^k)
    &=\limsup_{s\searrow0}\left[-\frac{\dd}{\dd s}\ent(u_s)\right]\\
    &\ge\int_\R \partial_x\big(\M(u_\tau^k)\partial_x\rho\big)^\tT 
    \big[\partial_x\nabla_pf(\partial_xu_\tau^k,u_\tau^k)-\nabla_zf(\partial_xu_\tau^k,u_\tau^k)\big]\dd x\\
    &=-\int_\R\nonl(u_\tau^k)[\rho]\dd x.
  \end{align*}
  Now insert this into \eqref{eq:potflowint}:
  \begin{align}
    \label{eq:weakform1}
    \begin{split}
      &\tau\int_\R\left[\rho^\tT\frac{u_\tau^k-u_\tau^{k-1}}{\tau} - \nonl(u_\tau^k)[\rho]\right]\dd x\\
      &\le \alpha\big(\calH(u_\tau^{k-1})-\calH(u_\tau^k)\big) + \frac{(-\lambda)}2\W_\M(u_\tau^k,u_\tau^{k-1})^2.      
    \end{split}
  \end{align}
  Each step in the derivation of \eqref{eq:weakform1} remains valid 
  upon replacing consistenly $\rho$ by $-\rho$ everywhere.
  Thus, the right-hand side in \eqref{eq:weakform1} is even a bound 
  on the \emph{absolute value} of the expression on the left-hand side.
  Next, we multiply the resulting inequality by the absolute value of $\psi_\tau^k:=\psi((k-1)\tau)$;
  since clearly $|\psi_\tau^k|\le\|\psi\|_{C^0}$, we obtain
  \begin{align}
    \label{eq:weakform2}
    \begin{split}
      &\left|\tau\int_\R\left[-\psi_\tau^k\rho^\tT\frac{u_\tau^k-u_\tau^{k-1}}{\tau}+ \psi_\tau^k\nonl(u_\tau^k)[\rho]\right]\dd x\right|\\
      &\le \|\psi\|_{C^0}\left[\alpha\big(\calH(u_\tau^{k-1})-\calH(u_\tau^k)\big) + \frac{(-\lambda)}2\W_\M(u_\tau^k,u_\tau^{k-1})^2\right].      
    \end{split}
  \end{align}
  Now we sum over $k=1,2,\ldots,N_\tau$,
  where $N_\tau$ is the smallest integer with $N_\tau\tau>T$;
  notice that on the left-hand side, we can actually sum to infinity,
  since $\psi(t)=0$ for all $t\ge T$.
  We apply the triangle inequality to the left-hand side,
  observe that
  \begin{align*}
    -\tau\sum_{k=1}^\infty \int_\R\rho^\tT\frac{u_\tau^k-u_\tau^{k-1}}{\tau}\psi_\tau^k \dd x
    = \tau\sum_{k=0}^\infty\int_\R \frac{\psi_\tau^{k+1}-\psi_\tau^k}\rho^\tT\tau u_\tau^k\dd x,
  \end{align*}
  and write the time summation over $n$ as an integral with respect to time, 
  ending up with
  \begin{align*}
    \begin{split}
      &\left|\int_0^\infty\int_\R
        \bigg[\frac{\psi_\tau(t+\tau)-\psi_\tau(\tau)}{\tau}\rho^\tT u_\tau(t) + \psi(t)\nonl\big(u_\tau(t)\big)[\rho]\bigg]\dd x\dd t\right|\\
      &\le \|\psi\|_{C^0}\left[\alpha\big(\calH(u_\tau^0)-\calH(u_\tau(N_\tau\tau))\big) + \frac{(-\lambda)}2\sum_{k=1}^\infty\W_\M(u_\tau^k,u_\tau^{k-1})^2\right].      
    \end{split}
  \end{align*}
  For simplification of the right-hand side, we use 
  the lower bound on $\calH$ from Theorem \ref{thm:HV} 
  in combination with the a priori estimate \ref{prp:apriori1}(d),
  and the energy estimate in Proposition \ref{prp:apriori1}(b).
  This proves the claim.
\end{proof}
\section{Passage to the continuous-time limit}
The proof of Theorem \ref{thm:exist4} can now be completed by passing to the \break continuous-time limit $\tau\searrow 0$. 
Using the a priori estimates from Propositions \ref{prp:apriori1} and \ref{prp:apriori2}, 
we first pass to a strong limit $u_\tau\to u$. 
The remainder of the proof then is concerned with obtaining the continuous weak formulation \eqref{eq:cweak} of system \eqref{eq:equation4} 
from the discrete weak formulation \eqref{eq:dweak} and the convergence properties of $(u_\tau)_{\tau>0}$. 
In summary, we have:
\begin{prop}[Continuous-time limit]\label{prop:ctl4}
  Let $T>0$ be given, $(\tau_k)_{k\in\N}$ be a vanishing sequence of step sizes, i.e. $\tau_k\searrow 0$ as $k\to\infty$, 
  and $( u_{\tau_k})_{k\in\N}$ be the corresponding sequence of discrete solutions obtained by the minimizing movement scheme \eqref{eq:mms}. 
  Then, there exists a (non-relabelled) subsequence and a limit curve $ u:\,[0,T]\to\Xaux$ such that as $k\to\infty$:
  \begin{enumerate}[(a)]
  \item $u_{\tau_k}(t)$ converges weakly$\ast$ to $u(t)$, pointwise with respect to $t\in[0,T]$;
  \item $ u_{\tau_k}-\zref\rightharpoonup  u-\zref$ weakly in $L^2([0,T];H^2(\R;\R^n))$;
  \item $ u_{\tau_k}-\zref\rightarrow  u-\zref$ strongly in $L^2([0,T];H^1_\loc(\R;\R^n))$,
  \end{enumerate}
  with the properties
  \begin{align*}
    u(0,\cdot)&=u^0,\\
    u&\in C^{1/2}([0,T];(\measm,\W_\M)),\\
    u-\zref&\in L^\infty([0,T];H^1(\R;\R^n)).
  \end{align*}
  Moreover, the limit $u$ is a weak solution to \eqref{eq:equation4} in the sense of Definition \ref{def:weaksol}.
\end{prop}
\begin{proof}
  For the convergence in (a), we use the refined version of the Arzel\`{a}-Ascoli theorem \cite[Thm. 3.3.1]{savare2008},
  which also provides the claimed H\"older continuity of $u$ with respect to $\W_\M$.
  The cited theorem is applicable thanks to the $\tau$-uniform quasi-H\"older continuity of the curves $u_\tau$ in the sense of Proposition \ref{prp:apriori1}(c)
  and the fact that the set of $u\in\measm$ at finite distance to the initial datum $u^0$ is a complete metric space with respect to $\W_\M$,
  see Theorem \ref{thm:propW}.
  In case (A), thanks to the $\tau$-uniform bound on the second momenta of the $u_\tau$, see Proposition \ref{prp:apriori1}(d),
  it follows that the second momenta of the limit $u$ obey the same bound,
  and moreover, by Prokhorov's theorem, the masses $\mss{u_\tau-\zref}$ are preserved in the limit $\tau\searrow0$.
  Thus, $u$ is a curve in $\Xaux$.

  Convergence (b) is an immediate consequences of the a priori estimate in Proposition \ref{prp:apriori2}(b) and Alaoglu's theorem.
  Further, from the time-monotonicity of the free energy $\ent$, see Proposition \ref{prp:apriori1}(a), 
  and the coercivity and lower semicontinuity of $\ent$ stated in Proposition \ref{prop:propE}, we infer $u-\zref\in L^\infty([0,T];H^1(\R;\R^n))$.

  The next step is to prove the strong convergence asserted in (c) using Theorem \ref{thm:ex_aub}. 
  Fix a bounded interval $I\subset\R$ and define
  \begin{align*}
    \ban&:=\{w\in\mathscr{M}(I;S):\,w-\zref\in H^1(I;\R^n)\},
  \end{align*}
  endowed with $\|w\|_\ban:=\|w-\zref\|_{H^1(I)}$, which is isometric to a closed subset of $H^1(I;\R^n)$ and the map
  \begin{align*}
    \mathcal{A}(u):=\begin{cases}\|u-\zref\|_{H^2(I)}^2&\text{if }u-\zref\in H^2(I;\R^n),\\ +\infty&\text{otherwise},\end{cases}
  \end{align*}
  which has relatively compact sublevels in $\ban$ due to Rellich's theorem. Furthermore, define the pseudo-distance $\W$ as
  \begin{align*}
    \W(w,\tilde w)&:=\inf\bigg\{\bigg.\W_\M(w',\tilde w'):~w',\tilde w'\in\measm,\\&\quad\W_\M(w',u^0)\le W,~\W_\M(\tilde w',u^0)\le W,~w'|_I=w,~\tilde w'|_I=\tilde w\bigg.\bigg\},
  \end{align*}
  where $W>0$ is the constant for which one has
  \begin{align*}
    \W_\M(u_\tau(t),u^0)&\le W,
  \end{align*}
  uniformly in $t\in[0,T]$ and $\tau\in(0,\bar\tau]$, recall Proposition \ref{prp:apriori1}(b). 
  Thanks to the properties of $\W_\M$ (cf. Theorem \ref{thm:propW}), 
  one easily sees that finiteness of the infimum above yields the existence of a minimizer, 
  and that the requirements for $\W$ of Theorem \ref{thm:ex_aub} are fulfilled. 
  We verify the hypotheses \eqref{eq:hypo1}\&\eqref{eq:hypo2} for the sequence $(u_k)_{k\in\N}$ defined by $u_k:=u_{\tau_k}|_{[0,\infty)\times I}$: 
  \eqref{eq:hypo1} is immediate because of the a priori estimate from Proposition \ref{prp:apriori2}(b). 
  For \eqref{eq:hypo2}, we first notice that 
  \begin{align*}
    \W(u_k(t+h),u_k(t))&\le \W_\M( u_{\tau_k}(t+h), u_{\tau_k}(t))
  \end{align*}
  by construction of $\W$.
  We claim that
  \begin{align}
    \label{eq:hypo2ver}
    \begin{split}
      &\sup_{k\in\N}\int_0^{T-h}\W_\M( u_{\tau_k}(t+h), u_{\tau_k}(t))\dd t\\&\le \max\left(1,\sqrt{T+\bar\tau}\right)\sqrt{2(\ent( u^0)-\inf\ent)(T+\bar \tau)h},
    \end{split}\end{align}
  for all $h\in(0,\bar\tau)$, from which \eqref{eq:hypo2} follows. 
  Indeed, let us fix $k\in\N$, and consider two cases. 
  If $h\in (0,\tau_k]$, then
  \begin{align*}
    \int_0^{T-h}\W_\M( u_{\tau_k}(t+h), u_{\tau_k}(t))\dd t
    &=\sum_{i=1}^{\gau{\frac{T}{\tau_k}}}h\W_\M( u_{\tau_k}^i, u_{\tau_k}^{i+1})\\
    &\le \sqrt{2(\ent( u^0)-\inf\ent)}\sqrt{h^2\gau{\frac{T}{\tau_k}}}\\
    &\le \sqrt{2(\ent( u^0)-\inf\ent)(T+\bar\tau)h},
  \end{align*}
  thanks to H\"older's inequality and Proposition \ref{prp:apriori1}(b). 
  If instead $h\in(\tau_k,\bar\tau]$, 
  then we directly get from Proposition \ref{prp:apriori1}(c):
  \begin{align*}
    \int_0^{T-h}\W_\M( u_{\tau_k}(t+h), u_{\tau_k}(t))\dd t
    &\le (T-h)\sqrt{2(\ent( u^0)-\inf\ent)h}\\
    &\le (T+\bar\tau)\sqrt{2(\ent( u^0)-\inf\ent)h}.
  \end{align*}
  Hence, \eqref{eq:hypo2ver} holds 
  and the application of Theorem \ref{thm:ex_aub} yields the existence of a (non-relabelled) subsequence 
  which converges in $H^1(I;\R^n)$ to (the spatial restriction to $I$ of) $ u$ in measure w.r.t. $t\in (0,T)$. 
  Employing the estimate from Proposition \ref{prp:apriori2}(a) and the dominated convergence theorem, 
  we conclude that
  \begin{align*}
    u_{\tau_k}-\zref\rightarrow  u-\zref \text{ strongly in }L^2([0,T];H^1(I;\R^n)),
  \end{align*}
  proving claim (c) for a prescribed interval $I$. 
  By a diagonal argument, setting $I_R:=[-R,R]$ and letting $R\nearrow \infty$, 
  we deduce that (c) is true simultaneously for every bounded interval $I$, extracting a further subsequence. 
  Moreover, extracting a further subsequence if necessary, $ u_{\tau_k}$ converges to $ u$ almost everywhere in $[0,T]\times \R$.

  It now remains to prove that the limit map $u$ fulfills the weak formulation \eqref{eq:cweak}. 
  Fix $\rho\in C^\infty_c(\R;\R^n)$ and $\psi\in C^\infty_c((0,T))$,
  and apply Lemma \ref{lemma:dweak} about the discrete weak formulation.
  Choosing $\alpha_k:=\sqrt{\tau_k}$ for each $k\in\N$,
  we obtain from \eqref{eq:dweak}:
  \begin{align*}
    &\Bigg|\Bigg.\int_0^\infty \int_\R \bigg[\bigg.\frac{\psi_{\tau_k}(t+\tau_k)-\psi_{\tau_k}(t)}{\tau_k}\rho^\tT u_{\tau_k}(t)
    +\psi_{\tau_k}(t) \nonl\big(u_{\tau_k}(t)\big)[\rho]\bigg.\bigg]\dd x\dd t\Bigg.\Bigg|\nonumber\\
    &\le\|\psi\|_{C^0}\left(\sqrt{\tau_k}[\calH(u^0)+M(1+T)]+C\tau_k\left(1+\frac1{\sqrt{\tau_k}}\right)\ent(u^0)\right)
    \le C'\sqrt{\tau_k},
  \end{align*}
  with a constant $C'$ that is independent of $\tau_k$,
  and thus the absolute value on the left-hand side converges to zero as $\tau_k\to0$.
  We wish to identify the limit inside the absolute value with the left-hand side of \eqref{eq:cweak},
  with $\varphi(t,x)=\psi(t)\rho(x)$.

  For the first integral, it suffices to observe that
  \begin{align*}
    \frac{\psi_{\tau_k}(t+\tau_k)-\psi_{\tau_k}(t)}{\tau_k} \to \partial_t\psi(t)
  \end{align*}
  uniformly with respect to $t\in[0,T]$, thanks to the smoothness of $\psi$.
  Indeed, since $\rho$ is smooth and of compact support, 
  we trivially have from the strong convergence of $u_{\tau_k}$ in $L^2([0,T];H^1_\loc(\R;\R^n))$ 
  that $\rho^\tT u_{\tau_k}\to \rho^\tT u$ in $L^2([0,T]\times\R;\R^n)$,
  and hence
  \begin{align*}
    \int_0^\infty \int_\R \frac{\psi_{\tau_k}(t+\tau_k)-\psi_{\tau_k}(t)}{\tau_k}\rho^\tT u_{\tau_k}(t)\dd x\dd t
    \to \int_0^\infty\int_\R \partial_t\psi(t)\rho^\tT u(t)\dd x\dd t.
  \end{align*}
  To show the convergence of the integal of $\nonl(u_{\tau_k})[\rho]$ to the one of $\nonl(u)[\rho]$,
  we recall that
  \begin{align*}
    \nonl(u_{\tau_k})[\rho] = -\partial_x\big(\M(u_{\tau_k})\partial_x\rho\big)^\tT Z_k
  \end{align*}
  with, using the abbreviations $\nabla_zf_k=\nabla_zf(\partial_xu_{\tau_k},u_{\tau_k})$ etc.,
  \begin{align*}
    Z_k= \partial_x\nabla_pf_k-\nabla_zf_k
    = \nabla_{pp}f_k\partial_x^2u_{\tau_k}+\nabla_{pz}f_k\partial_xu_{\tau_k}-\nabla_zf_k.
  \end{align*}
  Let us first discuss the convergence
  \begin{align}
    \label{eq:conv0}
    \partial_x\big(\M(u_{\tau_k})\partial_x\rho\big)\to \partial_x\big(\M(u_{\tau_k})\partial_x\rho\big) \quad\text{in $L^2([0,T]\times\R;\R^n)$}.
  \end{align}
  Since $z\mapsto\M(z)$ is a $C^2$-smooth function on the compact set $S$,
  the convergence of $u_{\tau_k}$ in $L^2([0,T];H^1_\loc(\R;\R^n))$ 
  induces convergence of $\M(u_{\tau_k})$ in the respective space $L^2([0,T];H^1_\loc(\R;\R^{n\times n}))$.
  Since $\rho$ is smooth and of compact support, we arrive at \eqref{eq:conv0}.

  Now, we study convergence of $Z_k$.
  As noted before, $(p,z)\mapsto\nabla_zf(p,z)$ is a globally Lipschitz continuous function on $\R^n\times S$,
  and both $(p,z)\mapsto\nabla_{pp}f(p,z)$ and $(p,z)\mapsto\nabla_{pz}f(p,z)$ are globally bounded.
  Hence, convergence of $u_{\tau_k}$ in $L^2([0,T];H^1_\loc(\R;\R^n))$ 
  implies on the one hand 
  \begin{align}
    \label{eq:conv1}
    \nabla_zf_k\to\nabla_zf(\partial_xu,u)\quad \text{in $L^2([0,T];L^2_\loc(\R;\R^n))$}.
  \end{align}
  And on the other hand, it implies pointwise convergence of $u_{\tau_k}$ and $\partial_xu_{\tau_k}$ 
  almost everywhere on $[0,T]\times\R$,
  and thereby also pointwise convergence of $\nabla_{pp}f_k$ and $\nabla_{pz}f_k$ 
  to their respective limits $\nabla_{pp}f(\partial_xu,u)$ and $\nabla_{pz}f(\partial_xu,u)$, 
  almost everywhere on $[0,T]\times\R$.
  Now we can make use of the weak convergence of $u_{\tau_k}$ in $L^2([0,T];H^2(\R;\R^n))$,
  which implies weak convergence of both $\partial_xu_{\tau_k}$ and $\partial_x^2u_{\tau_k}$ in $L^2([0,T]\times\R;\R^n)$:
  thanks to Lemma \ref{lem:fncana}, this allows to conclude
  \begin{align}
    \label{eq:conv2}
    &\nabla_{pp}f_k\partial_x^2u_{\tau_k}\rightharpoonup\nabla_{pp}f(\partial_xu,u)\partial_x^2u
    \quad \text{in $L^2([0,T]\times\R;\R^n)$},\\
    \label{eq:conv3}
    &\nabla_{pz}f_k\partial_xu_{\tau_k}\rightharpoonup\nabla_{pz}f(\partial_xu,u)\partial_xu
    \quad \text{in $L^2([0,T]\times\R;\R^n)$}.    
  \end{align}
  Summation of \eqref{eq:conv1}, \eqref{eq:conv2}, and \eqref{eq:conv3} yields that
  \begin{align*}
    Z_k\rightharpoonup 
    &\nabla_{pp}f(\partial_xu,u)\partial_x^2u+\nabla_{pz}f(\partial_xu,u)\partial_xu-\nabla_zf(\partial_xu,u) \\
    &=\partial_x \nabla_pf(\partial_xu,u)-\nabla_zf(\partial_xu,u)=:Z_0
  \end{align*}
  in each space $L^2([0,T];L^2(J;\R^n))$ for a compact interval $J\subset\R$.
  Since $\rho$ has compact support, so do all the expressions $\partial_x\big(\M(u_{\tau_k})\partial_x\rho\big)$,
  and we obtain for the product 
  \begin{align*}
    &\int_0^\infty\int_\R\nonl(u_{\tau_k})[\rho]\dd x\dd t
    =-\int_0^\infty\int_\R\partial_x\big(\M(u_{\tau_k})\partial_x\rho)^\tT Z_k\dd x\dd t\\
    &\longrightarrow
    -\int_0^\infty\int_\R\partial_x\big(\M(u)\partial_x\rho)^\tT Z_0\dd x\dd t
    =\int_0^\infty\int_\R\nonl(u)[\rho]\dd x\dd t.
  \end{align*}
  In summary, we have shown that the limit $u$ of $u_{\tau_k}$ satisfies the weak formulation \eqref{eq:cweak}
  for each test function $\varphi$ of the form $\varphi(t,x)=\psi(t)\rho(x)$ 
  with $\psi\in C^\infty_c((0,T))$ and $\rho\in C^\infty_c(\R)$, where $T>0$ is arbitrary.
  By standard approximation arguments, it follows that \eqref{eq:cweak} holds indeed for all $\varphi\in C^\infty_c(\R_{>0}\times\R)$.
\end{proof}

\section{Long time asymptotics}
We shall now prove Theorem \ref{thm:long}, 
which comes more or less as a corollary to the existence proof,
specifically from the a priori estimate in Lemma \ref{lem:H2}.

Let us start by proving the second claim in Theorem \ref{thm:long}, which is easier:
since $\partial_x^2 u_{\tau_k}$ converges weakly to $\partial_x^2 u$ in $L^2([0,T];L^2(\R;\R^n))$, 
and since the $L^2$-norm of $\partial_x^2u$ is a weakly lower semicontinuous functional,
the estimates in Lemma \ref{lem:H2} pass to the limit $\tau_k\searrow0$,
and yield, respectively,
\begin{align*}
  \int_0^T \|\partial_x^2u(t)\|_{L^2}^2\dd t &\le C_q(1+T^q), \quad\text{or}\\
  \int_0^T \|\partial_x^2u(t)\|_{L^2}^2\dd t &\le C,
\end{align*}
for all $T>0$,
where $q>\frac13$ is arbitrary.
The result now follows immediately from an abstract measure theoretic argument
that we recall for convenience in Lemma \ref{lem:measuretheory}. % in the Appendix.

The first claim in Theorem \ref{thm:long} is a bit more subtle.
Fix some $q\in(\frac13,1)$ and some $\tau>0$, 
and consider again the respective estimate in Lemma \ref{lem:H2}.
Given a time $T\ge1$, one necessarily finds some $t_\tau^T\le T$ such that
\[ \|\partial_x^2u_\tau(t_\tau^T)\|_{L^2}^2\le\frac{C_q(1+T^q)}{T}\le 2C_qT^{-(1-q)}. \]
With the help of the interpolation estimates \eqref{eq:interpol1} and \eqref{eq:interpol2} 
from Lemma \ref{lem:gagliardo} in the Appendix,
and recalling \eqref{eq:exaddon},
it now follows that
\begin{align*}
  &\|u_\tau(t_\tau^T)-\zref\|_{H^1}^2 
  = \sum_{j=1}^n\left(\|\partial_xu_{\tau,j}(t_\tau^T)\|_{L^2}^2+\|u_{\tau,j}(t_\tau^T)-\zref_j\|_{L^2}^2\right) \\
  &= \sum_{j=1}^n\left(
    \|\partial_x^2u_{\tau,j}(t_\tau^T)\|_{L^2}^{\frac65}\big[\mss{u^0-\zref}\big]_j^{\frac45}
    +\|\partial_x^2u_{\tau,j}(t_\tau^T)\|_{L^2}^{\frac25}\big[\mss{u^0-\zref}\big]_j^{\frac85}
  \right) \\
  &\le K\left(\|\partial_x^2u_\tau\|_{L^2}^{\frac65}+\|\partial_x^2u_\tau\|_{L^2}^{\frac25}\right)
  \le KC_q^{\frac65}T^{-\frac25(1-q)},
\end{align*}
where the constant $K$ depends only on $\mss{u^0-\zref}$.
Combining this with the estimate on $\ent$ in Proposition \ref{prop:propE}(a) 
and the monotonicity from Proposition \ref{prp:apriori1}(a),
we obtain
\begin{align*}
  \ent(u_\tau(T)) \le \ent(u_\tau(t_\tau^T)) \le KC_q^{\frac23}T^{-\frac25(1-q)}.
\end{align*}
This estimate is true for each $\tau>0$ and any $T\ge1$.
Thanks to the lower semi-continuity of $\ent$ from Proposition \ref{prop:propE}(b),
it passes to the limit along the sequence $\tau_k\searrow0$, and thus produces \eqref{eq:decay1},
with $p=\frac25(1-q)$.
Since any $q\in(\frac13,1)$ is admissible in the calculation above,
any $p>\frac4{15}$ can be attained.

\appendix

\section{Some technical lemmas}
In this section, we collect some small technical results that are presumably well-known,
but for which we were not able to find appropriate references.
\begin{lemma}
  \label{lem:fncana}
  Let $\Omega\subset\R^d$ be some open set.
  Assume that $(f_k)$ and $(g_k)$ are bounded sequences in $L^\infty(\Omega)$ and in $L^2(\Omega)$, respectively, 
  and assume that $f_k\to f$ pointwise almost everywhere,
  and $g_k\rightharpoonup g$ weakly in $L^2(\Omega)$.
  Then $f_kg_k\rightharpoonup fg$ weakly in $L^2(\Omega)$.
\end{lemma}
\begin{proof}
  Let $\eta\in L^2(\Omega)$ be arbitrary.
  Write
  \begin{align*}
    \eta(f_kg_k-fg) = (\eta f_k-\eta f)g_k + \eta f (g_k-g).
  \end{align*}
  Thanks to the hypotheses on $(f_k)$, 
  the product $\eta f_k$ converges to $\eta f$ strongly in $L^2(\Omega)$ by dominated convergence.
  Thus
  \begin{align*}
    &\lim_{k\to\infty}\left|\int_\Omega \eta(f_kg_k-fg)\dd x\right|\\
    &\le \lim_{k\to\infty}\|\eta f_k-\eta f\|_{L^2}\|g_k\|_{L^2} + \lim_{k\to\infty}\int_\Omega\eta f(g_k-g)\dd x
    = 0.
  \end{align*}
  This is true for every $\eta\in L^2(\Omega)$, hence $f_kg_k$ goes to $fg$ weakly in $L^2(\Omega)$.
\end{proof}
\begin{lemma}
  \label{lem:gagliardo}
  For each $f\in H^2(\R)\cap L^1(\R)$, we have
  \begin{align}
    \label{eq:interpol1}
    \|f\|_{L^2} &\le \|\partial_x^2f\|_{L^2}^{\frac15}\|f\|_{L^1}^{\frac45}, \\
    \label{eq:interpol2}
    \|\partial_xf\|_{L^2} &\le \|\partial_x^2f\|_{L^2}^{\frac35}\|f\|_{L^1}^{\frac25}.
  \end{align}
\end{lemma}
\begin{proof}
  As usual, it suffices to prove the statement for smooth functions $f$ of compact support.
  As an auxiliary result, we show that 
  \begin{align}
    \label{eq:interpolhelp1}
    \|f\|_{L^\infty}\le\|\partial_xf\|_{L^2}^{\frac12}\|f\|_{L^2}^{\frac12}.
  \end{align}
  By the fundamental theorem of calculus, an integration by parts, and the Cauchy-Schwarz inequality,
  one finds for each $a\in\R$ that
  \begin{align*}
    f(a)^2 
    &= \frac12\int_{-\infty}^a\partial_x(f^2)\dd x - \frac12\int_a^\infty\partial_x(f^2)\dd x \\
    & = \int_{-\infty}^af\partial_xf\dd x -  \int_a^\infty f\partial_x f\dd x \\
    &\le \|f\|_{L^2(\R_{<0})}\|\partial_xf\|_{L^2(\R_{<0})} +  \|f\|_{L^2(\R_{>0})}\|\partial_xf\|_{L^2(\R_{>0})} 
    \le\|\partial_xf\|_{L^2}\|f\|_{L^2}.
  \end{align*}
  With the help of \eqref{eq:interpolhelp1}, we conclude that
  \begin{align*}
    \|f\|_{L^2}^2 = \int_\R f^2\dd x \le \|f\|_{L^\infty}\|f\|_{L^1} 
    \le \|\partial_xf\|_{L^2}^{\frac12}\|f\|_{L^2}^{\frac12}\|f\|_{L^1}.
  \end{align*}
  Collecting powers of $\|f\|_{L^2}$ on the left-hand side yields another auxiliary estimate:
  \begin{align}
    \label{eq:interpolhelp2}
    \|f\|_{L^2}\le\|\partial_xf\|_{L^2}^{\frac13}\|f\|_{L^1}^{\frac23}.
  \end{align}
  As the last ingredient, we note that
  \begin{align}
    \label{eq:interpolhelp3}
    \|\partial_xf\|_{L^2}^2 = \int_\R (\partial_xf)^2\dd x 
    = -\int_\R f\,\partial_x^2f\dd x 
    \le \|\partial_x^2f\|_{L^2}\|f\|_{L^2}.
  \end{align}
  To prove \eqref{eq:interpol1}, 
  substitute the estimate \eqref{eq:interpolhelp3} for $\|\partial_xf\|_{L^2}$ into the right-hand side of \eqref{eq:interpolhelp2},
  then collect powers of $\|f\|_{L^2}$ on the left-hand side.
  To prove \eqref{eq:interpol2}, 
  substitute the estimate \eqref{eq:interpolhelp2} for $\|f\|_{L^2}$ into the last expression in formula \eqref{eq:interpolhelp2},
  then collect powers of $\|\partial_xf\|_{L^2}$.  
\end{proof}
\begin{lemma}
  \label{lem:measuretheory}
  Let a non-negative measurable function $f:\R_{>0}\to\R_{\ge0}$ be given.
  Assume that there are a $C>0$ and some $q<1$ such that
  \begin{align}
    \label{eq:mt1}
    \int_0^T f(t)\dd t \le C(1+T^q) \quad \text{for all $T>0$}.
  \end{align}
  Then, given $\delta>0$ and $p<q$, there is a measurable set $\Theta\subset\R_{>0}$ 
  with the following properites:
  \begin{enumerate}
  \item $\lebm{\Theta\cap\{t\le T\}}\le\delta T^p$ for all $T\ge0$, and
  \item $f(t)\to0$ for $t\to\infty$ and $t\in\R_{>0}\setminus\Theta$.
  \end{enumerate}
  If $q=0$, i.e., $f$ is integrable on $\R_{>0}$, then one can even take $p=0$ above,
  that is, the set $\Theta$ can be chosen of arbitrarily small Lebesgue measure.
\end{lemma}
\begin{proof}
  Let $\delta>0$ and $p<q$ be given.
  For each $n=1,2,\ldots$, define the time instance
  \begin{align*}
    T_n:=\max\left(1,\left(\frac{2C}{\delta}n^2(n+1)\right)^{\frac1{q-p}}\right),
  \end{align*}
  and the corresponding measurable set
  \begin{align*}
    S_n:=\left\{t\in\R_{>0}\,:\,f(t)\ge\frac1n,\,t\ge T_n\right\}.
  \end{align*}
  We have for every $T\ge T_n$ that
  \begin{align*}
    \frac1n\lebm{S_n\cap\{t\le T\}}
    &\le\int_0^Tf(t)\dd t \\
    &\le C(1+T^q) \le 2CT^q 
    \le \frac{2C}{T_n^{p-q}}T^p \le \frac1n\frac{\delta}{n(n+1)}T^p.
  \end{align*}
  On the other hand, for every $T<T_n$, 
  we have that $\lebm{S_n\cap\{t\le T\}}=0$ by definition of $S_n$.
  We thus conclude that
  \begin{align}
    \label{eq:Sn}
    \lebm{S_n\cap\{t\le T\}} \le \frac{\delta}{n(n+1)}T^p
  \end{align}
  holds for every $T>0$.
  Now define $\Theta$ as the union of all $S_n$.
  Adding up \eqref{eq:Sn} for all $n=1,2,\ldots$, we obtain
  \begin{align*}
    \lebm{\Theta\cap\{t\le T\}} \le \left(\sum_{n=1}^\infty\frac1{n(n+1)}\right)\delta T^p = \delta T^p.
  \end{align*}
  Let $(t_m)$ be any sequence with $t_m\to\infty$ and $t_m\in\R_{>0}\setminus\Theta$.
  By construction, $t_m>T_n$ implies that $f(t_m)<\frac1n$,
  and therefore $f(t_m)\to0$.

  In the case $q=0$, one can modify the definition of the $T_n$ in order to improve the result.
  Indeed, since $f$ is integrable on $\R_{>0}$, each set $S_n$ has finite Lebesgue measure,
  for any choice of $T_n$.
  By the properties of the Lebesgue measure, one can make $\lebm{S_n}$ smaller than any given positive value
  by choosing $T_n$ sufficiently large.
  In particular, we can choose $T_n$ such that
  \[ \lebm{S_n} \le \frac{\delta}{n(n+1)}, \]
  which replaces \eqref{eq:Sn} above.
\end{proof}

%%%%%%%%%%%%%%%%

\bibliography{lit-arxiv}

\end{document}